\documentclass[10pt,reqno]{amsart}

\usepackage{amsmath,amsfonts,amssymb,amsthm,mathtools,esint,color,hyperref,cleveref,subcaption}

\usepackage{amsmath,amsfonts,amssymb,amsthm,color}
\usepackage{mathrsfs}

\usepackage[utf8]{inputenc}
\usepackage{graphicx}

\usepackage{marginnote}

\theoremstyle{plain}
\newtheorem{theorem}{Theorem}[section]

\newtheorem{lemma}[theorem]{Lemma}
\newtheorem{proposition}[theorem]{Proposition}

\theoremstyle{definition}

\theoremstyle{remark}
\newtheorem*{remark*}{Remark}
\newtheorem{remark}[theorem]{Remark}

\numberwithin{equation}{section}
\newtheorem*{theorem*}{Theorem} 

\allowdisplaybreaks

\crefname{corollary}{Corollary}{Corollaries}
\crefname{proposition}{Proposition}{Propositions}
\crefname{lemma}{Lemma}{Lemmas}
\crefname{theorem}{Theorem}{Theorems}
\crefname{remark}{Remark}{Remarks}

\usepackage{enumerate}

\newcommand{\R}{{\mathbb R}}
\newcommand{\N}{{\mathbb N}}

\newcommand{\CM}{{\mathcal{M}}}
\newcommand{\CB}{{\mathcal{B}}}

\newcommand{\B}{{\mathcal{B}}}

\def\vint_#1{\mathchoice%
          {\mathop{\kern 0.2em\vrule width 0.6em height 0.69678ex depth -0.58065ex
                  \kern -0.8em \intop}\nolimits_{\kern -0.4em#1}}%
          {\mathop{\kern 0.1em\vrule width 0.5em height 0.69678ex depth -0.60387ex
                  \kern -0.6em \intop}\nolimits_{#1}}%
          {\mathop{\kern 0.1em\vrule width 0.5em height 0.69678ex
              depth -0.60387ex
                  \kern -0.6em \intop}\nolimits_{#1}}%
          {\mathop{\kern 0.1em\vrule width 0.5em height 0.69678ex depth -0.60387ex
                  \kern -0.6em \intop}\nolimits_{#1}}}
\def\vintslides_#1{\mathchoice%
          {\mathop{\kern 0.1em\vrule width 0.5em height 0.697ex depth -0.581ex
                  \kern -0.6em \intop}\nolimits_{\kern -0.4em#1}}%
          {\mathop{\kern 0.1em\vrule width 0.3em height 0.697ex depth -0.604ex
                  \kern -0.4em \intop}\nolimits_{#1}}%
          {\mathop{\kern 0.1em\vrule width 0.3em height 0.697ex depth -0.604ex
                  \kern -0.4em \intop}\nolimits_{#1}}%
          {\mathop{\kern 0.1em\vrule width 0.3em height 0.697ex depth -0.604ex
                  \kern -0.4em \intop}\nolimits_{#1}}}

\newcommand{\intav}{\vint}
\newcommand{\aveint}[2]{\mathchoice%
          {\mathop{\kern 0.2em\vrule width 0.6em height 0.69678ex depth -0.58065ex
                  \kern -0.8em \intop}\nolimits_{\kern -0.45em#1}^{#2}}%
          {\mathop{\kern 0.1em\vrule width 0.5em height 0.69678ex depth -0.60387ex
                  \kern -0.6em \intop}\nolimits_{#1}^{#2}}%
          {\mathop{\kern 0.1em\vrule width 0.5em height 0.69678ex depth -0.60387ex
                  \kern -0.6em \intop}\nolimits_{#1}^{#2}}%
          {\mathop{\kern 0.1em\vrule width 0.5em height 0.69678ex depth -0.60387ex
                  \kern -0.6em \intop}\nolimits_{#1}^{#2}}}

\def\XXint#1#2#3{{\setbox0=\hbox{$#1{#2#3}{\int}$}
\vcenter{\hbox{$#2#3$}}\kern-.5\wd0}}

\newcommand{\dy}{\, \mathrm{d}y}

\newcommand{\loc}{{\mbox{\scriptsize{loc}}}}

\newcommand\intd{\mathop{}\!\mathrm{d}}

\newcommand\tx\text
\newcommand\f\frac
\newcommand\avint\fint

\newcommand\cl[1]{\overline{#1}}

\usepackage{enumerate}

\pdfoutput=1

\usepackage[margin=1.4in]{geometry}

\title[Continuity gradient fractional maximal function]{Continuity of the gradient of the fractional maximal operator on $W^{1,1}(\R^d)$}
\author[D. Beltran, 
C. Gonz\'alez-Riquelme,
J. Madrid and
J. Weigt]{
David Beltran, 
Cristian Gonz\'alez-Riquelme,
Jos\'e Madrid and
Julian Weigt
}

\address{David Beltran: Department of Mathematics, University of Wisconsin, 480 Lincoln Drive, Madison, WI, 53706, USA.}
\email{dbeltran@math.wisc.edu}

\address{Cristian Gonz\'alez-Riquelme: IMPA - Instituto de Matematica Pura e Aplicada, Estrada Dona Castorina, 110, Jardim Botanico,
Rio de Janeiro - RJ, Brazil, 22460-320.}
\email{cristian@impa.br}	
	
\address{José Madrid: Department of  Mathematics,  University  of  California,  Los  Angeles (UCLA),  Portola Plaza 520, Los  Angeles,
California, 90095, USA}
\email{jmadrid@math.ucla.edu}	

\address{Julian Weigt:
Aalto University,
Department of Mathematics and Systems Analysis,
Otakaari 1,
Espoo,
Finland
}
\email{julian.weigt@aalto.fi}


\date\today

\begin{document}

\begin{abstract}
We establish that the map $f\mapsto |\nabla \CM_{\alpha}f|$ is continuous from $W^{1,1}(\mathbb{R}^d)$ to $L^{q}(\mathbb{R}^d)$, where $\alpha\in (0,d)$, $q=\frac{d}{d-\alpha}$ and $\CM_{\alpha}$ denotes either the centered or non-centered fractional Hardy--Littlewood maximal operator. In particular, we cover the cases $d >1$ and $\alpha \in (0,1)$ in full generality, for which results were only known for radial functions.
\end{abstract}

\maketitle

\section{Introduction}
Given $f\in L^1_{\loc}(\R^d)\,$ and $0\leq \alpha<d\,$, the centered fractional Hardy--Littlewood maximal operator is defined by
\begin{equation*}
M_{\alpha} f(x):=\sup_{r >0}
\frac{r^\alpha}{|B(x,r)|}\int_{B(x,r)}|f(y)|\dy
\end{equation*}
for every $x\in\R^d\,$. The non-centered version of $M_\alpha$, denoted by $\widetilde{M}_\alpha$, is defined by taking the supremum over all balls  \(B(z,r)\) such that \(x\) is contained in the closure of \(B(z,r)\). In what follows, we use $\CM_\alpha$ to denote either the centered or non-centered version, in the sense that if we formulate a result or a proof for $\CM_\alpha$, we mean that it holds for both $M_\alpha$ and $\widetilde{M}_\alpha$. The non-fractional case $\alpha=0$ corresponds to the classical maximal function, which we denote by $M=M_0$, $\widetilde{M}=\widetilde{M}_0$  and $\CM=\CM_0$.

The study of regularity properties for $\CM$ and $\CM_{\alpha}$ started with the influential works of Kinnunen \cite{Kinnunen1997} and Kinnunen and Saksman \cite{KS2003}, where it was established that 
\begin{equation}\label{kinnunensinequality}
    |\nabla \CM_\alpha f(x)| \leq  \CM_\alpha |\nabla f|(x)
\end{equation}
a.e.\ in $\R^d$. The mapping properties of $\CM_\alpha$ then imply that the map $f\mapsto \CM_{\alpha}f$ is bounded from $W^{1,p}(\mathbb{R}^d)$ to $W^{1,q}(\mathbb{R}^d)$ when $1<p\le d/\alpha$ and $\frac{1}{q}=\frac{1}{p}-\frac{\alpha}{d}$. At the endpoint $p=1$ this boundedness fails since $\CM_{\alpha}f\notin L^{q}(\mathbb{R}^d)$ unless $f =0$ a.e.. However, one can still consider the following question:
\begin{equation}\label{mapfractional}
\textrm{ Is the map   $f \mapsto |\nabla \CM_\alpha f|$  bounded from $W^{1,1}(\mathbb{R}^d)$ to $L^{\frac{d}{d-\alpha}}(\mathbb{R}^d)$ ?}
\end{equation}
By a dilation argument, this is equivalent to proving that there exists a constant $C>0$ such that
\begin{equation}\label{sobolev boundedness}
    \| \nabla \CM_\alpha f \|_{L^{d/(d-\alpha)}(\R^d)} \leq C \| \nabla f \|_{L^1(\R^d)}. 
\end{equation}

 This question was first explored in the classical case $\alpha=0$ and $d=1$ \cite{Tanaka2002,AP2007,Kurka2010} and, more recently, for $d>1$, radial functions and non-centered $\widetilde{M}$ \cite{Luiro2017}. For $\alpha>0$, this boundedness was first considered in \cite{CM2015}, where the case $d=1$ was settled for $\widetilde{M}_\alpha$. Moreover, they observed that the case $d>1$, $1 \leq \alpha < d $ follows via Sobolev embedding and the smoothing property of $\CM_{\alpha}$ obtained by Kinnunen and Saksman \cite{KS2003}, which ensures, that 
if \(1\leq\alpha<d\) and $f \in L^p(\R^d)$ with $1 \leq p \leq d/\alpha$,
then
\begin{equation}\label{KS}
|\nabla \CM_\alpha f(x)| \leq  (d-\alpha) \CM_{\alpha-1} f(x)
\end{equation}
a.e. in $\R^d$. 

For $0 < \alpha<1$, the first boundedness result in higher dimensions was established for $\widetilde M_{\alpha}$ in \cite{LM2017} for radial functions. Analogous results in both $d=1$ and $d>1$ were obtained for $M_{\alpha}$ in \cite{BM2019.2}, where a pointwise relation between $\nabla M_{\alpha}$ and $\nabla \widetilde M_{\alpha}$ was observed for the first time for \(\alpha>0\). That relation revealed that both operators behave quite similarly, unlike it was previously thought; note that without taking the gradient the two maximal functions are comparable.
Very recently, the question \eqref{mapfractional} was established in full generality by the fourth author in \cite{Weigt} for $\alpha>0$, completing the remaining open cases in the fractional setting (that is, $d>1$, $0 < \alpha < 1$ and general $f$).
He originally proved it for the uncentered operator \(\widetilde M_\alpha\),
but he observed shortly after that almost the same proof also works for the centered operator \(M_\alpha\), see \cite[Remark~1.3]{Weigt}.
The proof in \cite{Weigt} is based on the corresponding bound for the dyadic maximal operator in the non-fractional case \(\alpha=0\) in \cite{weigtdyadic}.
Other interesting related results in the context of fractional maximal functions have recently been proven in \cite{BRS2018,GR,HKKT2015,RSW,Saari2016}.

In this manuscript we explore the continuity of the map \(f\mapsto|\nabla\CM_\alpha f|\) for $\alpha>0$. Note that this map is not sublinear, and thus its boundedness from $W^{1,1}(\R^d)$ to $L^{d/(d-\alpha)}(\R^d)$ does not immediately imply its continuity as a map between those function spaces. For $p>1$, the continuity can be established by the methods developed by Luiro \cite{Luiro2007}, which rely on the Lebesgue space mapping properties of $\CM_\alpha$. Once again, the endpoint case $p=1$ is more intricate. For $d=1$ the continuity was established by the third author \cite{Madrid2017} for the non-centered case and by the first and third authors \cite{BM2019} for the centered case. For $d>1$, similarly to the boundedness, we shall distinguish between the ranges $1 \leq \alpha < d$ and $0 < \alpha < 1$. For the former range, the result can be obtained via the inequality \eqref{KS} and dominated convergence theorem arguments. This was proven in \cite{BM2019}. The range $0 < \alpha < 1$ is harder as one can no longer appeal to \eqref{KS}. 
Positive results under a radial assumption on $f$ were obtained by the first and third authors in both the non-centered \cite{BM2019} and centered case \cite{BM2019.2}. We refer to \cite{CMP2017,CGRM,grk,Luiro2018} for complementary results regarding the continuity of $\widetilde{M}$. 

%In this direction, it is also interesting to mention the result by the second author \cite{GR} for a polar variant of $M_\alpha$ on $\mathbb{S}^d$. \db{how to cite this?}

%Here we establish the continuity for $\alpha>0$ in the remaining open cases, that is, $d>1$, $0 < \alpha < 1$ and general functions $f$. Our proof does not use any of these additional restrictions, so it leads to the following complete result for $\alpha>0$. 
Here we establish the following complete result for \(\alpha>0\), which in particular yields the continuity in the remaining open cases, that is, for $d>1$, $0 < \alpha < 1$ and general functions $f \in W^{1,1}(\R^d)$.

\begin{theorem}\label{thm:main}
Let $\mathcal{M}_\alpha \in \{\widetilde{M}_\alpha, M_\alpha\}$. If $0 < \alpha < d$, the operator $f \mapsto |\nabla \mathcal{M}_\alpha f|$ maps continuously $W^{1,1}(\R^d)$ into $L^{d/(d-\alpha)}(\R^d)$.
\end{theorem}

As observed in \cite{BM2019}, it suffices to establish the continuity for any compact set $K \subset \R^d$. For any given \(\delta>0\), we consider two types of points in $K$, depending on whether the ball with maximal average has 
\textit{large} radius (larger than \(\delta\)) or
\textit{small} radius (smaller than \(\delta\)).
The techniques from \cite{BM2019,BM2019.2} immediately apply to prove the continuity for the points whose maximal ball has \textit{large} radius: the radiality assumption was not used in that situation. 

Thus, in order to establish continuity in Theorem \ref{thm:main}, it suffices to bound contributions coming from points whose maximal ball has
\textit{small} radius, i.e.\ radius smaller than \(\delta\), and show that they go to zero for \(\delta\rightarrow0\).
This is the main novelty of this paper.
To obtain this bound for points with \textit{small} radius, we first note that on any compact set \(K\), \(\CM_\alpha f\) is bounded away from \(0\).
Then we use the Poincaré--Sobolev inequality, which becomes stronger the smaller the radius is and the larger the average of the function is.
Then we apply a refined version of \eqref{KS} which allows us to invoke a local version of the boundedness \eqref{sobolev boundedness} in \cite{Weigt} on the subset of points with \textit{small} radius. This yields the desired result. In the passage, we also use a refined version of \eqref{kinnunensinequality}.

The proof of Theorem \ref{thm:main} is presented in Section \ref{sec:proof}. Auxiliary results which feature prominently in the proof are presented in Sections \ref{2} and \ref{sec:convergences}.

\subsubsection*{Notation} Given a measurable set $E \subseteq \R^d$, %, $\chi_E$ denotes the characteristic function of $E$ and
we denote by $E^c:=\R^d \backslash E$ the complementary set of $E$ in $\R^d$. For $c \in \R$, we denote by $cE$ the concentric set to $E$ dilated by $c$. The integral average of $f \in L^1_{\loc}(\R^d)$ over $E$ is denoted by $f_E\equiv \intav_E f:= |E|^{-1} \int_E f$. Given a ball $B \subseteq \R^d$, we denote its radius by $r(B)$. The volume of the $d$-dimensional unit ball is denoted by $\omega_d$. The weak derivative of $f$ is denoted by $\nabla f$.

\subsubsection*{Acknowledgments}

The authors would like to thank Juha Kinnunen for his encouragement. %\db{complete}
D.B.\ was partially supported by NSF grant DMS-1954479.
J.W.\ has been supported by the Vilho, Yrj\"o and Kalle V\"ais\"al\"a Foundation of the Finnish Academy of Science and Letters.

\section{Families of good balls}\label{2}
In this section we develop some estimates and identities regarding the weak derivative of the maximal functions of interest. We shall only be concerned with $0 < \alpha < d$, although many of the arguments can also be extended to $\alpha=0$.

\subsection{The truncated fractional maximal function}

An important object for our purposes are the truncated fractional maximal operators which, for a given $\delta>0$, are defined as
$$
M_\alpha^\delta f (x):= \sup_{r > \delta} r^\alpha \intav_{B(x,r)} |f(y)| \dy \qquad \text{ and } \qquad \widetilde{M}_\alpha^\delta f (x):= \sup_{\substack{ \bar B(z,r) \ni x \\ r > \delta}} r^\alpha \intav_{B(z,r)} |f(y)|  \dy.
$$
We use $\CM_\alpha^\delta$ to denote either $M_\alpha^\delta$ or $\widetilde{M}_\alpha^\delta$.
Note that if $\delta=0$, we recover the original operators $\CM_\alpha = \CM_\alpha^0$. The following is a well-known and elementary result; see for instance \cite[Lemma 2.4]{BM2019.2} and \cite[Lemma 8]{HM2010}.

\begin{proposition}
Let $0 < \alpha < d$ and $\delta >0$. If $f \in L^1 (\R^d)$, then $\CM_\alpha^\delta f$ is Lipschitz continuous (in particular, a.e.\ differentiable).
\end{proposition}

\subsection{Weak derivative and approximate derivative}\label{subsec:app der}

As mentioned in the introduction, the fourth author proved in \cite{Weigt}, after partial contributions by many, the following result.

\begin{theorem}[{\cite[Theorem~1.1 and Remark~1.3]{Weigt}}]\label{thm:Julian}
Let $0 < \alpha < d$ and $f \in W^{1,1}(\R^d)$. Then $\CM_\alpha f$ is weakly differentiable and there exists a constant \(C_{d,\alpha}>0\) such that
\[
\|\nabla \CM_\alpha f\|_{L^{d/(d-\alpha)}(\R^d)}
\leq C_{d,\alpha}
\| \nabla f\|_{L^1(\R^d)}.
\]
\end{theorem}

It will be convenient in our arguments to also recall the concept of approximate derivative. A function $f: \mathbb{R}^d \to \mathbb{R}$ is said to be {\it approximately differentiable} at a point $x_0\in \mathbb{R}$ if there exists a vector $Df(x_0)\in\R^d$ such that, for any $\varepsilon >0$, the set
\begin{equation}\label{eq:app der}
A_{\varepsilon} := \left\{ x \in \mathbb{R} : \ \frac{|f(x) - f(x_0) - \langle Df(x_0),x- x_0\rangle|}{|x-x_0|} < \varepsilon \right\}
\end{equation}
has $x_0$ as a density point. In this case, $Df(x_0)$ is called the \textit{approximate derivative} of $f$ at $x_0$ and it is uniquely determined. It is well-known that if $f$ is weakly differentiable, then $f$ is approximate differentiable a.e. and the weak and approximate derivatives coincide \cite[Theorem 6.4]{EvansGariepy}.

The approximate derivative satisfies the following property, which will play a rôle in Propositions \ref{lemma:derivative Malpha} and \ref{pro_ks} below.

\begin{lemma}
\label{lem:appderseq}
Let $f$ be approximately differentiable at a point $x \in \R^d$. Then there exists a sequence $\{h_n\}_{n \in \N}$ with $|h_n| \to 0$ such that
\begin{equation*}
    |Df(x)|= - \lim_{n \to \infty} \frac{ f(x+h_n) - f(x)}{|h_n|},
\end{equation*}
where $Df(x)$ denotes the approximate derivative of $f$ at $x$.
\end{lemma}

\begin{proof}
Let \(0<\varepsilon<\pi/2\).
By the definition of the approximate derivative,
there exists $0< \rho<\varepsilon$ such that
\begin{equation}\label{eq_approxdiff}
|A_\varepsilon \cap B(0,\rho)|
\geq
\Bigl(
1-
\f{\omega_{d-1}}{d\, \omega_d}(\sin\varepsilon)^{d-1}(\cos\varepsilon)^d
\Bigr)
|B(0,\rho)|
\end{equation}
where $A_\varepsilon$ is as in \eqref{eq:app der}.

If $D f(x)=0$, the result simply follows by the definition of $A_\varepsilon$ and taking $\varepsilon=1/n$.

Assume next $Df(x) \neq 0$. 
For each $h \in \R^d$, let $\beta_h$ denote the angle formed by $h$ and $-D f (x)$, so that
\begin{equation*}
    - \langle Df(x) , h \rangle = |Df(x)||h| \cos \beta_h.
\end{equation*}
\begin{figure}
\includegraphics{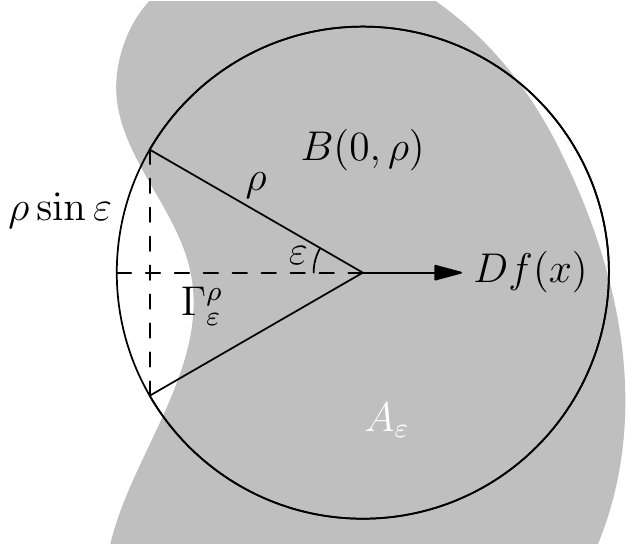}
\caption{The sets \(\Gamma_{\varepsilon,\rho}\) and \(A_\varepsilon\) intersect.}
\end{figure}
The set
\begin{equation*}
    \Gamma_{\varepsilon,\rho} := \{ h \in B(0,\rho): \beta_h \leq \varepsilon \}
\end{equation*}
has measure 
\[
|\Gamma_{\varepsilon,\rho}|
>
\int_0^{\rho\cos\varepsilon}
\omega_{d-1}(r\sin\varepsilon)^{d-1}
\intd r
=
\f{\omega_{d-1}}d(\sin\varepsilon)^{d-1}(\cos\varepsilon)^d\rho^d
.
\]
Thus, it follows from \eqref{eq_approxdiff} that $\Gamma_{\varepsilon,\rho} \cap A_{\varepsilon} \neq \emptyset$, so by the definition of $A_\varepsilon$ there is an \(h\in\R^d\) such that
\begin{equation}\label{eq:properties h}
\f{
| f(x+h)- f(x)-\langle Df(x),h\rangle|
}
{|h|}
<\varepsilon,
\qquad  \beta_h \leq \varepsilon \qquad \text{and} \qquad 0 <|h|<\rho<\varepsilon.
\end{equation}
By the triangle inequality, for $h$ satisfying \eqref{eq:properties h},
\begin{align*}
\biggl|
| Df(x)|
+
\f{
 f(x+h)- f(x)
}
{|h|}
\biggr|
& \leq
\biggl|
|Df(x)|
+
\f{
\langle Df(x),h\rangle
}
{|h|}
\biggr|
+
\biggl|
\f{
 f(x+h)- f(x)
}
{|h|}
-
\f{
\langle Df(x),h\rangle
}
{|h|}
\biggr|
\\
& <
|Df(x)| |1 - \cos \beta_h | + \varepsilon
\\
& \leq
|Df(x)| |1 - \cos \varepsilon | + \varepsilon.
\end{align*}
As $|Df(x)| \neq 0$, the result now follows taking $\varepsilon = \min \{ 1/2n, 1/ \sqrt{|Df(x) n} \}$  and the corresponding \(h_n=h\) from the previous display.
\end{proof}

The approximate derivative of $M f$ for a.e.\ approximately differentiable functions $f \in L^1(\R^d)$ was studied by Haj\l asz and Maly \cite{HM2010}. In particular, their arguments show that if $f \in L^1$ is a.e.\ approximately differentiable, then $\CM_\alpha f$ is a.e.\ approximately differentiable.

\subsection{The families of good balls}\label{subsec:good balls}

Let $0 < \alpha < d$ and $\delta\geq0$. For the uncentered maximal operator, given a function $f \in W^{1,1}(\R^d)$ and a point $x \in \R^d$, define the family of \textit{good balls} for $f$ at $x$ as
$$
\CB_{\alpha,x}^\delta \equiv \CB_{\alpha,x}^{\delta} (f):=\Big\{ B(z,r): r\geq\delta,\ x \in \overline{B(z,r)}, \: \CM_{\alpha}^{\delta} f (x)= r^\alpha \intav_{B(z,r)} |f(y)| \dy \Big\}.
$$
For the centered maximal operator we use the same definition, except that \(z=x\).
Note that $\CB_{\alpha,x}^\delta  \neq \emptyset$ for all $x \in \R^d$ if $\delta>0$.
Moreover, by the Lebesgue differentiation theorem $\CB_{\alpha,x} \equiv \CB_{\alpha, x}^0  \neq \emptyset$ for a.e.\ $x \in \R^d$, and if $B(z,r) \in \CB_{\alpha,x}^0$, then $r>0$. This immediately implies that for a.e.\ $x$ there exists $\delta_x>0$ such that if $0 \leq \delta<\delta_x$, then
\begin{equation*}
    M_\alpha^\delta f(x)=M_\alpha f(x).
\end{equation*}
This type of observation will be used at the derivative level in the forthcoming Lemma \ref{delta convergence}.

\subsection{Luiro's Formula}
An important tool for our purposes is the so called Luiro's formula, which relates the derivative of the maximal function with the derivative of the original function. It corresponds to a refinement of Kinnunen's inequality \eqref{kinnunensinequality} and has its roots in \cite[Theorem 3.1]{Luiro2007}. 

\begin{proposition}\label{lemma:derivative Malpha}
Let $0 < \alpha < d$, $\delta \geq 0$ and $f \in W^{1,1}(\R^d)$. Then, for a.e.\ $x \in \R^d$ and $B=B(z,r) \in \CB^\delta_{\alpha,x}$, the weak derivative $\nabla \CM^\delta_{\alpha}f$ satisfies 
\begin{equation}\label{eq:Luiro}
\nabla \CM_\alpha^\delta f(x) = r^\alpha \intav_{B} \nabla |f|(y) \dy.
\end{equation}
\end{proposition}
\begin{proof} 
This essentially follows from an argument of Haj\l asz and Maly \cite[Theorem 2]{HM2010}, which we include for completeness. 
By \S\ref{subsec:app der} the weak gradient of \(\CM_\alpha^\delta f\) equals its approximate gradient almost everywhere, so it suffices to show \eqref{eq:Luiro} at a point $x$ at which $\CM_\alpha^\delta f$ is approximately differentiable and for which there exists $B=B(z_x,r_x) \in \CB_{\alpha,x}^\delta$.
Define the function $\varphi: \R^d \to \R$ by
\begin{equation*}
    \varphi(y):= \CM_\alpha^\delta f (y) - r^\alpha \intav_{B(z_x + y - x, r_x)} |f(t)| \, \mathrm{d} t = \CM_\alpha^\delta f (y) - r^\alpha \intav_{B(z_x - x, r_x)} |f(y+t)| \, \mathrm{d} t,
\end{equation*}
which satisfies $\varphi \geq 0$ and $\varphi(x) =0$. Thus, $\varphi$ has a minimum at $x$. Furthermore, $\varphi$ is approximately differentiable at $x$ (note that one can differentiate under the integral sign) and by Lemma \ref{lem:appderseq} there exists a sequence $\{h_n\}_{n \in \N}$ with $|h_n| \to 0$ such that
\begin{equation*}
    |D\varphi(x)| = - \lim_{n \to \infty} \frac{\varphi(x+h_n)-\varphi(x)}{|h_n|}.
\end{equation*}
As $\varphi$ has a minimum at $x$, the right-hand side is non-positive and thus $D \varphi(x)=0$, which yields the desired result.
\end{proof}

\begin{remark}
Proposition \ref{lemma:derivative Malpha} continues to hold for $\alpha=0$, replacing the weak derivative by the approximate derivative in the cases where the weak differentiability of $\CM$ is currently unknown.
\end{remark}

\subsection{Refined Kinnunen--Saksman Inequality}

The Kinnunen--Saksman inequality \eqref{KS} admits a refinement in terms of the good balls, in the same spirit as Luiro's formula \eqref{eq:Luiro} improves over Kinnunen's pointwise inequality \eqref{kinnunensinequality}. It is noted that further refinements involving boundary terms (that is, averages along spheres) have been obtained in \cite{LM2017} and \cite{BM2019.2} for $\widetilde{M}_\alpha$ and $M_\alpha$ respectively, although these are not required for the purposes of this paper.

\begin{proposition}
\label{pro_ks}
Let $0<\alpha<d$, $\delta\geq0$ and $f \in W^{1,1}(\R^d)$.
Then, for a.e.\ \(x \in \R^d\) and $B=B(z,r) \in \mathcal{B}_{x,\alpha}^\delta$, the weak derivative $\nabla \CM_\alpha^\delta f$ satisfies
\begin{equation}\label{eq_ks}
|\nabla \CM_{\alpha}^\delta f(x)|
\leq(d-\alpha) r^{\alpha-1} \intav_B |f(y)| \, \mathrm{d} y.
\end{equation}
\end{proposition}

\begin{proof} 
By \S\ref{subsec:app der} the weak gradient of \(\CM_\alpha^\delta f\) equals its approximate gradient almost everywhere, so it suffices to show \eqref{eq_ks} at a point $x$ at which $\CM_\alpha^\delta f$ is approximately differentiable and for which there exists $B=B(z_x,r_x) \in \CB_{\alpha,x}^\delta$.
By Lemma \ref{lem:appderseq} there is a sequence \(\{h_n\}_{n \in \N}\) with \(|h_n|\rightarrow0\)
and
\[
|\nabla\CM_\alpha^\delta f(x)|
=
\lim_{n\rightarrow\infty}
\f{
\CM_\alpha^\delta f(x)-\CM_\alpha^\delta f(x+h_n)
}
{|h_n|}
.
\]
Now the proof follows from the classical Kinnunen--Saksman \cite{KS2003} reasoning, which we include for completeness.
Note that \(x+h_n\in\cl{B(z+h_n,r+|h_n|)}\),
and that for the centered maximal operator we have \(z=x\). This implies
\begin{equation*}
    \CM_\alpha^\delta f(x+h_n) \geq  (r+|h_n|)^\alpha \intav_{B(z+h_n, r+|h_n|)} |f(y)|\dy.
\end{equation*}
Therefore
\begin{align*}
&\f{
\CM_\alpha^\delta f(x)-\CM_\alpha^\delta f(x+h_n)
}
{|h_n|}
\\
& \qquad \leq
\f1{\omega_d|h_n|}
\Big(r^{\alpha-d}\int_{B(z,r)}|f(y)|\dy-(r+h_n)^{\alpha-d}\int_{B(z+h_n,r+|h_n|)}|f(y)|\dy\Big)
\\
& \qquad \leq
\f1{\omega_d|h_n|}
\Big(r^{\alpha-d}\int_{B(z+h_n,r+|h_n|)}|f(y)|\dy-(r+|h_n|)^{\alpha-d}\int_{B(z+h_n,r+|h_n|)}|f(y)|\dy\Big)
\\
& \qquad =
\f
{r^{\alpha-d}-(r+{|h_n|})^{\alpha-d}}
{\omega_d|h_n|}
\int_{B(z+h_n,r+{|h_n|})}|f(y)|\dy
\\
& \qquad \rightarrow
\f
{(d-\alpha)r^{\alpha-d-1}}
{\omega_d}
\int_{B(z,r)}|f(y)|\dy
\end{align*}
for \(n\rightarrow\infty\), which concludes the proof.
\end{proof}

\begin{remark}
Proposition \ref{pro_ks} continues to hold for $\alpha=0$, replacing the weak derivative by the approximate derivative in the cases where the weak differentiability of $\CM$ is currently unknown.
\end{remark}

\subsection{A refined fractional maximal function}

In view of the Kinnunen--Saksman type inequality \eqref{eq_ks}, it is instructive to define the operator
\[
\CM_{\alpha,-1} f(x) = \sup_{B\in\B_{\alpha,x}(f)}r(B)^{\alpha-1} \intav_B |f(y)|\dy,
\]
so that for any $0 < \alpha < d$,
\begin{equation}
\label{eq:maximal Julian 0}
    |\nabla \CM_{\alpha} f(x)| \leq (d-\alpha) \CM_{\alpha,-1} f(x) \qquad \text{for a.e. $x \in \R^d$}.
\end{equation}
Furthermore, this extends to the case $\delta>0$, that is,
\begin{equation}
\label{eq:maximal Julian delta}
    |\nabla \CM_{\alpha}^\delta f(x)| \leq (d-\alpha) \CM_{\alpha,-1} f(x) \qquad \text{for a.e. $x \in \R^d$}.
\end{equation}
Indeed, let $\delta>0$ and $B \in \mathcal{B}_{\alpha,x}^\delta$. Then, there exists $C \in \mathcal{B}_{\alpha,x}$ such that $r(C) \leq r(B)$. This immediately yields
\begin{equation*}
    r(B)^{\alpha-1} \intav_B |f| \leq r(C)^{\alpha-1} \intav_C |f| \leq \CM_{\alpha,-1}f(x),
\end{equation*}
which implies \eqref{eq:maximal Julian delta} via Proposition \ref{pro_ks}.

The proof of Theorem \ref{thm:Julian} in \cite{Weigt} is obtained through the analogous bound on $\CM_{\alpha,-1}$. Indeed, such a bound is of local nature.
The following is a local version of \cite[Theorem~1.2]{Weigt}; see \cite[Remark~1.9]{Weigt}.
\begin{theorem}
\label{theo_mfdr}
Let $0 < \alpha < d$ and $E \subseteq \R^d$. There exist constants $c>1$ and $C_{d,\alpha}>0$ such that the inequality
\[
\|\CM_{\alpha,-1} f\|_{L^{d/(d-\alpha)}(E)}
\leq C_{d,\alpha}
\| \nabla f\|_{L^1(D)}
\]
holds for all $f \in W^{1,1}(\R^d)$, where \begin{equation*}
    D=\bigcup_{B \in \mathcal{I}_E} cB \qquad \text{ and } \qquad  \mathcal{I}_E:=\{ B \in \mathcal{B}_{\alpha,x} \, ; \, \text{for some } \,x \in E\}.   
\end{equation*}
\end{theorem}

\begin{remark}
For $0 < \alpha < d$ one has,  combining \eqref{eq:maximal Julian 0} and \cref{theo_mfdr}, that
\begin{equation*}
\|\nabla\CM_\alpha f\|_{L^{d/(d-\alpha)}(E)}
\leq (d-\alpha) \, C_{d,\alpha}
\| \nabla f\|_{L^1(D)},
\end{equation*}
where $C_{d,\alpha}$ is the constant in Theorem \ref{theo_mfdr}.
\end{remark}

\subsection{Poincar\'e--Sobolev Inequality}
Another important tool for our purposes is the following.
\begin{lemma}
\label{lem_poincare}
Let $0 < \alpha< d$, $f\in W^{1,1}(\R^d)$, $x\in\R^{d}$, $B=B(z,r)\in \CB_{\alpha,x}(f)$ % be a good ball
and \(c>1\).
Then there is a constant \(C_{d,\alpha,c}\) such that
\[
\intav_{cB} |f(y)| \dy
\leq C_{d,\alpha,c}
\, r
\intav_{cB}|\nabla f(y)|\dy
.
\]
\end{lemma}
\begin{proof}
By the triangle inequality and the Poincar\'e-Sobolev inequality there is a \(C_d\) such that
\[
\intav_{cB}\big||f(y)|-|f|_{cB}\big| \dy \leq
\intav_{cB}|f(y)-f_{cB}|\dy
\leq C_d \, 
r
\intav_{cB}|\nabla f(y)| \dy
.
\]
Since $B\in \CB_{\alpha,x}$ we have \(c^\alpha |f|_{cB}<|f|_B\).
This and the triangle inequality yield
\[
c^{d}\intav_{cB}\big||f(y)|-|f|_{cB}\big| \dy
\geq
\intav_B\big||f(y)|-|f|_{cB}\big| \dy
\geq
 |f|_B-|f|_{cB}
\geq
(c^{\alpha}-1)\intav_{cB} |f(y)| \dy.
\]
%so that by the Poincaré-Sobolev inequality
Then, combining the above, we obtain
\[
\intav_{cB} |f(y)| \dy
\leq
\frac{c^dC_d}{c^\alpha-1}
r
\intav_{cB}|\nabla f(y)|\dy,
\]
as desired.
\end{proof}

\section{Convergences}\label{sec:convergences}

In this section we review some auxiliary convergence results established in the series of papers \cite{CMP2017, BM2019} which reduce the proof of Theorem \ref{thm:main} to the convergence of the difference $\CM_\alpha f_j-\CM_\alpha^\delta f_j$ on a compact set.

\subsection{A Sobolev space lemma}

We start recalling an auxiliary result concerning the convergence of the modulus of a sequence in $W^{1,1}(\R^d)$. This is useful in view of the identity \eqref{eq:Luiro}. 

\begin{lemma}[{\cite[Lemma 2.3]{BM2019}}]\label{lemma:convergence modulus in W11}
Let $f \in W^{1,1}(\R^d)$ and $\{f_j\}_{j \in \N} \subset W^{1,1}(\R^d)$ be such that $\| f_j - f \|_{W^{1,1}(\R^d)} \to 0$ as $j \to \infty$. Then $\big\| |f_j| - |f| \big\|_{W^{1,1}(\R^d)} \to 0$ as $j \to \infty$.
\end{lemma}

\subsection{Convergence outside a compact set}

By Theorem \ref{thm:Julian} and the work of the first and third author in \cite{BM2019} we have that it suffices to study the convergence in a compact a set.
\begin{proposition}[{\cite[Proposition 4.10]{BM2019}}]\label{prop:smallness}
Let $0	<  \alpha < d$, $f \in W^{1,1}(\R^d)$ and $\{f_j\}_{j \in \N} \subset W^{1,1}(\R^d)$ such that $\| f_j - f \|_{W^{1,1}(\R^d)} \to 0$. Then, for any $\varepsilon>0$ there exists a compact set $K$ and $j_\varepsilon>0$ such that
$$
\| \nabla \CM_\alpha f_j - \nabla \CM_\alpha f \|_{L^{d/(d-\alpha)}( ( 3K)^c )} < \varepsilon
$$
for all $j \geq j_\varepsilon$.
\end{proposition}

\subsection{Continuity of \texorpdfstring{$\CM_\alpha^\delta$}{M a d} in \texorpdfstring{$W^{1,1}(\mathbb{R}^d)$, $\delta>0$}{W11(Rd), d>0}}

A key observation is the a.e.\ convergence of the maximal function $\CM_\alpha^\delta f_j$ at the derivative level.

\begin{lemma}\label{lemma:ae convergence derivatives}
Let $0 <\alpha < d$, $ \delta \geq 0$, $f \in W^{1,1}(\R^d)$ and $\{f_j\}_{j \in \N} \subset W^{1,1}(\R^d)$ be such that $\| f_j - f \|_{W^{1,1}(\R^d)} \to 0$ as $j \to \infty$. Then
\begin{equation*}
\nabla \CM_\alpha^\delta f_j (x) \to \nabla \CM_\alpha^\delta f(x) \quad \text{a.e.} \quad as \:\:j \to \infty.
\end{equation*}
\end{lemma}

A version of this result for the full $\CM_\alpha$ is given in \cite[Lemma 2.4]{BM2019}. The proof for $\CM_\alpha^\delta$ is identical (in fact, it slightly simplifies), and relies on Luiro's formula for $\CM_\alpha^\delta$, that is, \cref{lemma:derivative Malpha}. We omit further details. For $\delta>0$, we have the following norm convergence.

\begin{proposition}\label{thm:large radious}
Let $0 < \alpha < d$, $\delta>0$, $f \in W^{1,1}(\R^d)$ and $\{f_j\}_{j \in \N} \subset W^{1,1}(\R^d)$ be such that $\| f_j - f \|_{W^{1,1}(\R^d)} \to 0$ as $j \to \infty$.
Let $K\subset\R^d$ be a compact set. 
$$
\| \nabla \CM^{\delta}_{\alpha}f-\nabla \CM_{\alpha}^{\delta}f_j\|_{L^{d/(d-\alpha)}(K)}\to0
\ \text{as}\ j\to \infty.$$
\end{proposition}

\begin{proof}
By Proposition \ref{lemma:derivative Malpha} and Lemma \ref{lemma:convergence modulus in W11} there exists $j_0 \in \N$ such that
$$
|\nabla \CM_\alpha^\delta f_{j}(x)|\leq \frac{1}{\omega_d \, \delta^{d-\alpha}}\|\nabla |f_j|\|_1\leq \frac{1}{\omega_d \, \delta^{d-\alpha}}\|\nabla |f|\|_1+1 \qquad \text{ for all $j \geq j_0$}, \quad    \text{a.e.} \,\, x \in K.
$$
Furthermore, by Lemma \ref{lemma:ae convergence derivatives}
$$
\nabla \CM_\alpha^\delta f_j (x) \to \nabla \CM_\alpha^\delta f(x) \quad \text{a.e.} \quad as \:\:j \to \infty.
$$
The convergence on $L^{d/(d-\alpha)}(K)$ then follows from the dominated convergence theorem.
\end{proof}

\subsection{\texorpdfstring{$\delta$}{d}-convergence of \texorpdfstring{$\nabla \CM_\alpha^\delta f$}{grad Mad f}} Here we establish that $\nabla \CM_{\alpha}^{\delta}f$ provides a good approximation for $\nabla \CM_{\alpha}f$ in $L^{d/(d-\alpha)}(K)$ when $\delta\to 0$. This relies on the Theorem \ref{thm:Julian}.
\begin{lemma}\label{delta convergence}
Let $0<\alpha<d$ and $f\in W^{1,1}(\R^d)$. Then
\begin{equation*}
    \|\nabla \CM_\alpha f - \nabla \CM_{\alpha}^{\delta} f\|_{L^{d/(d-\alpha)}(K)}\to0 \qquad  \text{as}\quad  \delta\to0.
\end{equation*}
\begin{proof}
Recall from \S\ref{subsec:good balls} that for a.e.\ $x\in\R^d$ one has that if $B(z,r) \in \CB_{\alpha,x}^\delta$, then $r>0$. 
This and Luiro's formula \eqref{eq:Luiro} imply that for a.e.\ $x \in \R^d$ there exists $\delta_x>0$ such that  
$$
\nabla \CM^{\delta}_{\alpha}f(x)=\nabla\CM_{\alpha}f(x) \qquad \text{for all $0 \leq \delta<\delta_x$},
$$
and thus $\nabla \CM^{\delta}_{\alpha}f(x) \to \nabla\CM_{\alpha}f(x)$ for a.e.\ $x \in \R^d$ as $\delta \to 0$. Furthermore, as proven in \eqref{eq:maximal Julian delta}, for a.e.\ $x\in\R^d$ we have that
$$
|\nabla \CM^{\delta}_{\alpha}f(x)|\leq \CM_{\alpha,-1}f(x) \qquad \text{ for all $\delta\geq 0$.}
$$
Since $f \in W^{1,1}(\R^d)$, Theorem \ref{theo_mfdr} ensures that $\CM_{\alpha,-1}f\in L^{d/(d-\alpha)}(\mathbb{R}^d)$ and we can then conclude the result by the dominated convergence theorem.
\end{proof}
\end{lemma}

\section{Proof of Theorem \ref{thm:main}}\label{sec:proof}
%What is this {4}?

Let $f\in W^{1,1}(\R^d)$ and $\{f_j\}_{j \in \N}\subset W^{1,1}(\R^d)$ be a sequence of functions such that $\|f_j-f\|_{W^{1,1}(\R^d)}\to 0$ as $j\to\infty$. If $f=0$ then the result follows directly from the boundedness, that is \cref{thm:Julian}. From now on we assume that $f\neq0$.
Let \(\varepsilon>0\).
Then by Proposition \ref{prop:smallness} it is sufficient to prove that there exists $j^* \in \N$ such that
\begin{equation}\label{supergoal}
    \|\nabla \CM_\alpha f - \nabla \CM_\alpha f_j\|_{L^{d/(d-\alpha)}(K)} < 3\varepsilon
\end{equation}
for all $j \geq j^*$.
To this end, for any $\delta>0$, use the triangle inequality to bound
\begin{align}\label{triangle inequality}
\|\nabla \CM_\alpha f - \nabla \CM_\alpha f_j\|_{L^{d/(d-\alpha)}(K)}
&\leq
\|\nabla \CM_\alpha f - \nabla \CM_{\alpha}^{\delta} f\|_{L^{d/(d-\alpha)}(K)}
\\
&\qquad+
\|\nabla \CM_{\alpha}^{\delta} f - \nabla \CM_{\alpha}^{\delta} f_j\|_{L^{d/(d-\alpha)}(K)}
\notag \\
&\qquad+
\|\nabla \CM_{\alpha}^{\delta} f_j - \nabla \CM_\alpha f_j\|_{L^{d/(d-\alpha)}(K)}. \notag
\end{align}
To finish the proof, it suffices to show that for $\varepsilon>0$ fixed, there exist a \(\delta^*\) and a \(j^*\)
such that for \(\delta = \delta^*\) and all \(j\geq j^*\),
each of the summands on the right hand side of \eqref{triangle inequality} is bounded by \(\varepsilon\).
We choose $\delta^*$ depending on \(\varepsilon, K\) and \(f\), and $j^*$ depending on $\delta^*$, $\varepsilon$, $K$, $f$ and the sequence $\{f_j\}_{j \in \N}$.

For the first term, we know by Lemma \ref{delta convergence} that there exists a $\delta'>0$ such that
$$
\|\nabla \CM_\alpha f - \nabla \CM_\alpha^\delta f\|_{L^{d/(d-\alpha)}(K)}<\varepsilon
$$
for all $0 \leq \delta\leq\delta'$. For the second term, we have by Proposition \ref{thm:large radious}
that for every $\delta>0$ there exists a $j(\delta) \in \N$ such that
$$
\|\nabla \CM_{\alpha}^{\delta} f - \nabla \CM_{\alpha}^{\delta} f_j\|_{L^{d/(d-\alpha)}(K)}
<\varepsilon
$$
for all $j\geq j(\delta)$. 
The rest of the section is devoted to proving a favourable bound for the third term. More precisely, we will show that there are \( \tilde \delta>0\) and \(\tilde j\in\N\) such that for all \(0 \leq \delta\leq \tilde \delta\) and \(j\geq \tilde j\), 
\begin{equation}\label{goal third term}
    \|\nabla \CM_{\alpha}^{\delta} f_j - \nabla \CM_\alpha f_j\|_{L^{d/(d-\alpha)}(K)} < \varepsilon.
\end{equation}
Temporarily assuming this, we can then conclude that for \(\delta= \delta^*:=\min\{\delta', \tilde \delta\}\) and \(j\geq j^*:=\max\{j(\delta^*), \tilde j\}\), the right-hand side of \eqref{triangle inequality} is bounded by at most \(3\varepsilon\), as desired for \eqref{supergoal}.

We now turn to the proof of \eqref{goal third term}.
We start by noting that there exists a
$\lambda_0>0$ and a $j_0 \in \N$
such that for all $j\geq j_0$ and \(x\in K\) we have $\CM_\alpha f_j(x) > \lambda_0$. Indeed, as $f \in L^1(\R^d)$, there exists a ball $B_0$ that contains $K$ with 
\(
\int_{B_0} |f|
> \frac{1}{2}
\int_{\R^d} |f|
.
\)
As $\|f_j - f \|_1 \to 0$ as $j \to 0$, by the triangle inequality, there exists \(j_0>0\) such that for all \(j\geq j_0\) we have
\(
\int_{B_0} |f_j|
>
\frac{1}{2}\int_{B_0} |f| > \frac{1}{4} \int_{\R^d} |f|
.
\)
Then, for every \(j\geq j_0\) and \(x\in K\) we have
\begin{equation*}
\CM_\alpha f_j (x) \geq 2^{\alpha} r(B_0)^{\alpha} \intav_{B(x, 2r(B_0))} |f_j|  > \frac{(2r(B_0))^{\alpha-d}}{4\, \omega_d}   \int_{\R^d} |f|, 
\end{equation*}
where in the last inequality we have used that $B(x, 2r(B_0)) \supset B_0$ for all $x \in K$. Thus, we can take $\lambda_0$ to be the right-hand side of the inequality above.
Furthermore by \cref{lemma:derivative Malpha},
if there exists a $B\in\CB_{\alpha,x}(f_j)$ such that $r(B)\geq\delta$ then $\nabla\CM_\alpha f_j(x)=\nabla\CM_{\alpha}^{\delta} f_j(x)$.
Define
\begin{equation*}
    E_{\lambda_0, \delta, j}:=\Big\{ x \in K: \text{ if } B \in \mathcal{B}_{\alpha,x}(f_j), \text{ then} \,\,\,\,  r(B) < \delta \,\,\,\, \text{and} \,\,\,\,  r(B)^\alpha \intav_B |f_{j}|>\lambda_0
    \Big\}.
\end{equation*}
By the previous two observations, Proposition \ref{pro_ks} and a crude application of the triangle inequality, one has 
\begin{align*}
\|\nabla \CM_{\alpha}^{\delta} f_j - \nabla \CM_\alpha f_j\|_{L^{d/(d-\alpha)}(K)}  & =\|\nabla \CM_\alpha^\delta f_j - \nabla \CM_{\alpha} f_j\|_{L^{d/(d-\alpha)}(E_{\lambda_0,\delta,j})} \\
& \leq 2 (d-\alpha) \, \| \CM_{\alpha, -1} f_j \|_{L^{d/(d-\alpha)}(E_{\lambda_0,\delta,j})}.
\end{align*}
for all $j \geq j_0$.
Define the indexing set
\begin{equation*}
    \mathcal{I}_{\lambda_0,\delta,j} := \Big\{B \in \mathcal{B}_{\alpha, x}(f_j)  \, : x \in K, \,\, r(B) < \delta \,\, \text{ and } \,\,  r(B)^\alpha \intav_B |f_j| >\lambda_0 \Big\}
\end{equation*}
and consider the set
\begin{equation*}
    D_{\lambda_0,\delta,j}:=\bigcup_{B \in \mathcal{I}_{\lambda_0,\delta,j} } cB,
\end{equation*}
where \(c\) is the constant from \cref{theo_mfdr}.
Then, by Theorem \ref{theo_mfdr}, we have
\[
\|\CM_{\alpha,-1}f_j\|_{L^{d/(d-\alpha)}(E_{\lambda_0,\delta,j})}
\leq C_{d,\alpha}
\|\nabla f_j\|_{L^1(D_{\lambda_0,\delta,j})}
\]
for any $\delta>0$. 
Thus, the proof of \eqref{goal third term} is reduced to showing that there exist a $\tilde \delta>0$ and a $j_1\in\N$ such that for all $j\geq j_1$ and \(0 \leq \delta \leq \tilde \delta\) we have
\begin{equation}\label{smallness f_j}
    \| \nabla f_j \|_{L^1(D_{\lambda_0, \delta, j})} < \frac{\varepsilon}{2(d-\alpha)C_{d,\alpha}},
\end{equation}
as one can then take $\tilde{j}:=\max \{j_0,j_1\}$.

In order to prove \eqref{smallness f_j}, we first use the triangle inequality and that $\| \nabla f_j -\nabla f\|_{L^1(\R^d)} \to 0$ as $j \to \infty$ to find a $j_2 \in \N$ such that
\begin{equation}\label{f_j to f}
\|\nabla f_j\|_{L^1(D_{\lambda_0,\delta,j})} \leq \|\nabla f\|_{L^1(D_{\lambda_0,\delta,j})} + \frac{\varepsilon}{4(d-\alpha)C_{d,\alpha}}.
\end{equation}
for any $\delta>0$ and $j \geq j_2$.

Next, let $x\in D_{\lambda_0,\delta,j}$.
Then there is a $B \in \mathcal{I}_{\lambda_0, \delta, j}$ with \(x\in cB\).
So, by \cref{lem_poincare}, we have
\[
\lambda_0\leq c^d r(B)^\alpha \intav_{cB} |f_j|
 \leq C_{d,\alpha, c} \, c^{d+1} r(B)^{\alpha+1} \intav_{cB}|\nabla f_j|
 \leq C_{d,\alpha, c}\, c^{d-\alpha+1} \delta \, \widetilde{M}_\alpha|\nabla f_j|(x),
\]
where $\widetilde M_\alpha$ in the above inequality  denotes the uncentered fractional maximal operator.
Hence, by the weak $(1,d/(d-\alpha))$ inequality for $\widetilde{M}_\alpha$,
\begin{align}
|D_{\lambda_0,\delta,j}|
& \leq
\biggl|\biggl\{x : \widetilde{M}_\alpha|\nabla f_j|(x)\geq\frac{\lambda_0}{C_{d,\alpha,c}c^{d-\alpha+1}\delta}\biggr\}\biggr| \notag \\
&
\leq C_{d,\alpha,c,\lambda_0}
\delta^{d/(d-\alpha)}\|\nabla f_j\|_1^{d/(d-\alpha)}
\notag \\
&\leq C_{d,\alpha,c,\lambda_0}
\delta^{d/(d-\alpha)}\left(1+\|\nabla f\|_1^{d/(d-\alpha)}\right) \label{size D set}
\end{align}
if $j\geq j_3$ for some $j_3\in\N$, using that $\| \nabla f_j - \nabla f \|_{L^1(\R^d)} \to 0$ as $j \to \infty$.

Finally, note that as $\nabla f \in L^1(\R^d)$, there exists \(\rho>0\) 
such that for all \(A\subseteq\R^d\) satisfying \(|A|<\rho\), one has
\begin{equation}\label{app by simple}
\| \nabla f \|_{L^1(A)}
<
\frac{\varepsilon}{4(d-\alpha) C_{d,\alpha}}. 
\end{equation}
As the right-hand side of \eqref{size D set} goes to zero for \(\delta\rightarrow0\) uniformly in \(j\), there exists $\tilde{\delta}>0$ such that $|D_{\lambda_0,\delta,j}|< \rho$ for all $j \geq j_3$ and $\delta < \tilde{\delta}$. Thus, taking $j_1:=\max \{j_2,j_3\}$,  \eqref{smallness f_j} follows from combining \eqref{f_j to f} and \eqref{app by simple} with $A=D_{\lambda_0, \delta, j}$. This implies the claimed inequality \eqref{goal third term} and therefore finishes the proof of \cref{thm:main}. \qed

\begin{remark*}
Note that in the above proof, instead of using \cref{delta convergence} to bound the first term in \eqref{triangle inequality},
we could have also bounded it running the same scheme as for the third term.
\end{remark*}

\bibliography{Reference}
\bibliographystyle{amsplain}

\end{document}